\documentclass[11pt]{article}
\usepackage[includeheadfoot,margin=1in]{geometry}
\usepackage{times}
\usepackage{cite}
\usepackage{amsmath,amsfonts,amssymb,amsthm}
\usepackage{mathtools}
\usepackage{enumerate}
\usepackage{verbatim}
\usepackage{commath}
    \usepackage[usenames]{color}
    \definecolor{plum}  {rgb}{.4,0,.4}
    \definecolor{BrickRed} {rgb}{0.6,0,0}
	\definecolor{DarkBlue} {rgb}{0,0,0.6}
\usepackage[plainpages=false,pdfpagelabels,colorlinks=true,linkcolor=BrickRed,citecolor=plum,urlcolor=DarkBlue]{hyperref}
\usepackage[mathcal]{euscript}
\usepackage{cleveref}

\def\ddefloop#1{\ifx\ddefloop#1\else\ddef{#1}\expandafter\ddefloop\fi}

\def\ddef#1{\expandafter\def\csname b#1\endcsname{\ensuremath{\boldsymbol{#1}}}}
\ddefloop ABCDEFGHIJKLMNOPQRSTUVWXYZabcdefghijklmnopqrstuvwxyz\ddefloop

\def\ddef#1{\expandafter\def\csname c#1\endcsname{\ensuremath{\mathcal{#1}}}}
\ddefloop ABCDEFGHIJKLMNOPQRSTUVWXYZ\ddefloop
 
\def\ddef#1{\expandafter\def\csname s#1\endcsname{\ensuremath{\mathsf{#1}}}}
\ddefloop ABCDEFGHIJKLMNOPQRSTUVWXYZ\ddefloop

\def\Reals{{\mathbb R}}

\def\eps{\varepsilon}

\def\Naturals{{\mathbb N}}
\def\trn{{\mathsf T}} 

\makeatletter
\newsavebox{\@brx}
\newcommand{\llangle}[1][]{\savebox{\@brx}{\(\m@th{#1\langle}\)}%
  \mathopen{\copy\@brx\kern-0.5\wd\@brx\usebox{\@brx}}}
\newcommand{\rrangle}[1][]{\savebox{\@brx}{\(\m@th{#1\rangle}\)}%
  \mathclose{\copy\@brx\kern-0.5\wd\@brx\usebox{\@brx}}}
\makeatother

\newtheorem{theorem}{Theorem}

\newtheorem{proposition}{Proposition}
\newtheorem{lemma}{Lemma}

\newtheorem{remark}{Remark}

\begin{document}

\title{Some Remarks on Controllability of the Liouville Equation}
\author{Maxim Raginsky\thanks{University of Illinois, Urbana, IL 61801, USA} \\ \href{mailto:maxim@illinois.edu}{maxim@illinois.edu} }
\date{}
\maketitle

\begin{abstract} We revisit the work of Roger Brockett on controllability of the Liouville equation, with a particular focus on the following problem: Given a smooth controlled dynamical system of the form $\dot{x} = f(x,u)$ and a state-space diffeomorphism $\psi$, design a feedback control $u(t,x)$ to steer an arbitrary initial state $x_0$ to $\psi(x_0)$ in finite time. This formulation of the problem makes contact with the theory of optimal transportation and with nonlinear controllability. For controllable linear systems, Brockett showed that this is possible under a fairly restrictive condition on $\psi$. We prove that controllability suffices for a much larger class of diffeomorphisms. For nonlinear systems defined on smooth manifolds, we review a recent result of Agrachev and Caponigro regarding controllability on the group of diffeomorphisms. A corollary of this result states that, for control-affine systems satisfying a bracket generating condition, any $\psi$ in a neighborhood of the identity can be implemented using a time-varying feedback control law that switches between finitely many time-invariant flows. We prove a quantitative version which allows us to describe the implementation complexity of the Agrachev--Caponigro construction in terms of a lower bound on the number of switchings.
\end{abstract}

In a series of papers \cite{Brockett_1997,Brockett_2007,Brockett_2012}, Roger Brockett drew attention to a class of problems involving a smooth controlled dynamical system $\dot{x} = f(x,u)$, where the focus is not on the evolution of the system state $x(t)$ per se, but rather on the evolution of its probability density $\rho(t,\cdot)$ starting from a given initial density $\rho(0,\cdot) = \rho_0(\cdot)$. This shift of perspective leads to questions pertaining to control of the so-called \textit{Liouville equation} \cite{Mackey_1992}, i.e., the first-order partial differential equation
\begin{align}\label{eq:Liouvlle}
	\frac{\partial \rho(t,x)}{\partial t} = - {\rm div} \big(\rho(t,x) f(x,u) \big),
\end{align}
also referred to as the transport or the continuity equation. For example, we may be interested in the questions of controllability, where two densities $\rho_0$ and $\rho_1$ are given and the objective is to determine whether they can be joined by a curve lying along the trajectory of \eqref{eq:Liouvlle} for some choice of the control $u(\cdot)$. One can also formulate optimal control problems in this setting, e.g., minimizing a finite-horizon performance index of the form
\begin{align*}
	J(\rho_0; u(\cdot)) = \int^T_0 \int_M L(x,u,t)\rho(t,x)\dif x \dif t + \int_M r(x)\rho(T,x)\dif x,
\end{align*}
over appropriately chosen controls $u(\cdot)$, where the initial density $\rho(0,\cdot) = \rho_0(\cdot)$ is given and the integration is over both the (finite) time interval $[0,T]$ and the state space $M$. As pointed out by Brockett, because of the presence of the divergence operator on the right-hand side of \eqref{eq:Liouvlle}, there will in general be a nontrivial difference between what can be achieved with open-loop controls $u(t)$ versus closed-loop (or feedback) controls $u(t,x)$. See, e.g., \cite{Pogodaev_2016,Bartsch_2019} for further developments of these ideas.

In this paper, we will focus on the questions of controllability of \eqref{eq:Liouvlle}. For the most part, we will stick with the Euclidean setting $M = \Reals^n$ and assume that the target density $\rho_1$ can be expressed as the pushforward of the initial density $\rho_0$ by an orientation-preserving diffeomorphism $\psi : \Reals^n \to \Reals^n$, i.e., 
\begin{align*}
	\int_{\Reals^n} h(x)\rho_1(x)\dif x = \int_{\Reals^n} h(\psi(x))\rho_0(x)\dif x
\end{align*}
for all bounded, continuous functions $h : \Reals^n \to \Reals$. In this case, the two densities are related via the Monge--Amp\`ere equation
\begin{align*}
	\det D\psi(x) = \frac{\rho_0(x)}{\rho_1(\psi(x))},
\end{align*}
where $D\psi$ denotes the Jacobian of $\psi$. It is important to note that not any pair of densities $\rho_0, \rho_1$ can be joined in this way; for example, it will not be possible if $\rho_0$ vanishes on some open set while $\rho_1$ does not. Nevertheless, this restricted setting already exposes several nontrivial aspects of the problem of controlling the Liouville equation, and an added benefit is that in this instance we can also phrase everything in terms of the original system $\dot{x} = f(x,u)$ and ask whether there exists a feedback control law $u(t,x)$ that transfers the state of the system from $x(0) = x_0$ to $x(1) = \psi(x_0)$, for every $x_0$. This, in turn, makes contact with the problem of implementing orientation-preserving diffeomorphisms using a given controlled system \cite{Agrachev_2009,Caponigro_2011}. On the other hand, the question of existence of a diffeomorphism $\psi$ that pushes $\rho_0$ forward to $\rho_1$ makes contact with the theory of optimal transportation \cite{Villani_2003}. This aspect had not been explored by Brockett, apart from a brief discussion of the classical theorem of Moser \cite{Moser_1965} on the transitivity of the action of orientation-preserving diffeomorphisms of a compact differentiable manifold $M$ on the space of everywhere positive probability densities on $M$.

The remainder of the paper is structured as follows. In Section~\ref{sec:transport}, we take a quick look at controllability in the space of densities through the lens of optimal transportation. We then consider the case of controllable linear systems in Section~\ref{sec:linear}, where we present an extension of a controllability result from \cite{Brockett_2007} and relate it to optimal transportation. Moving on to nonlinear systems, in Section~\ref{sec:diffeo} we present a quantitative version of a structural result of Agrachev and Caponigro \cite{Agrachev_2009} on controllability on the group of diffeomorphisms of a smooth compact manifold. Some concluding remarks are presented in Section~\ref{sec:conclusion}.

\section{The relation to optimal transportation}
\label{sec:transport}

Let $\rho_0$ and $\rho_1$ be two probability densities on $\Reals^n$, which we will assume to be everywhere positive to keep things simple. Then a theorem of Brenier \cite[Thm.~2.12]{Villani_2003} guarantees the existence of a strictly convex function $h : \Reals^n \to \Reals$, such that its gradient $\psi = \nabla h$ pushes $\rho_0$ forward to $\rho_1$. Moreover, this mapping is optimal in the sense that
\begin{align*}
	\int_{\Reals^n} |x - \nabla h(x)|^2 \rho_0(x)\dif x = \min \left\{ \int_{\Reals^n} |x - \psi(x)|^2 \rho_0(x)\dif x : \rho_1 = \psi_* \rho_0 \right\},
\end{align*}	
where the minimum is over all $\psi : \Reals^n \to \Reals^n$ that push $\rho_0$ forward to $\rho_1$. There is also a complementary dynamic formulation due to Benamou and Brenier \cite[Sec.~8.1]{Villani_2003}; in control-theoretic language, it guarantees the existence of a smooth feedback control law $u : [0,1] \times \Reals^n \to \Reals^n$, such that the trajectory of the Liouville equation
\begin{align*}
	\frac{\partial \rho(t,x)}{\partial t} = -{\rm div}\big(\rho(t,x)u(t,x)\big)
\end{align*}
joins $\rho(0,\cdot) = \rho_0(\cdot)$ to $\rho(1,\cdot) = \rho_1(\cdot)$, and  the performance index
\begin{align*}
	J(\rho_0; u(\cdot)) = \frac{1}{2}\int^1_0 \int_{\Reals^n} |u(t,x)|^2 \rho(t,x)\dif x \dif t
\end{align*}
is minimized. On the other hand, the corresponding controlled system $\dot{x} = u$ allows for maximum ``control authority,'' in the sense that all directions of motion are available to the controller at each time $t \in [0,1]$. The more difficult case of a general controlled system $\dot{x} = f(x,u)$ corresponds to the so-called \textit{nonholonomic} setting \cite{Agrachev_2009a,Khesin_2009}, where the key question is how to relate the question of controllability (or obstructions to controllability) in the space of densities to structural properties of the family of vector fields $\{f(\cdot,u)\}_{u \in U}$ indexed by the elements of the control set $U$. 

\section{The case of a controllable linear system}
\label{sec:linear}

The case of a linear system
\begin{align}\label{eq:LTI}
	\dot{x} = Ax + Bu
\end{align}
with $n$-dimensional state $x$ and $m$-dimensional input $u$ has been considered  in \cite{Brockett_2007} under the assumption that the system is controllable, i.e., the columns of $B,AB,\dots,A^{n-1}B$ span $\Reals^n$.  Let an orientation-preserving diffeomorphism $\psi : \Reals^n \to \Reals^n$ be given. By controllability, we can steer \eqref{eq:LTI} from any \textit{fixed} initial state $x(0) = x_0$ to $x(T) = \psi(x_0)$ using the open-loop control
\begin{align*}
	u_{x_0}(t) = B^\trn e^{-A^\trn t} W(0,T)^{-1}\big(e^{-AT}\psi(x_0) - x_0 \big),
\end{align*}
where
\begin{align*}
	W(0,t) := \int^t_0 e^{-A\tau}BB^\trn e^{-A^\trn \tau}\dif \tau
\end{align*}
is the controllability Gramian of \eqref{eq:LTI}, which is positive definite since the system is controllable. The main idea in \cite{Brockett_2007} is to impose appropriate regularity conditions on $e^{-AT}\psi$ to ensure that the map
\begin{align}\label{eq:Kt}
	\begin{split}
	K_t(x_0) &:= e^{At}\left(x_0 + \int^t_0 e^{-A\tau}Bu_{x_0}(\tau)\dif \tau\right) \\
	&= e^{At}\left(x_0 + W(0,t)W(0,T)^{-1}\big(e^{-AT}\psi(x_0)-x_0\big)\right)
	\end{split}
\end{align}
is injective for all $0 < t < T$. If this is the case, then the closed-loop (feedback) control
\begin{align}\label{eq:ut_x}
	u(t,x) = B^\trn e^{-A^\trn t} W(0,T)^{-1}\big(e^{-AT}\psi(K^{-1}_t(x)) - K^{-1}_t(x) \big)
\end{align}
can be used to carry out the transfer from $x(0) = x_0$ to $x(T) = \psi(x_0)$ for \textit{every} initial condition $x_0 \in \Reals^n$.

In \cite{Brockett_2007}, Brockett stated that one sufficient condition to ensure the above result is for the map $e^{-AT}\psi - {\rm id}$ to be a contraction (i.e., Lipschitz-continuous with constant strictly smaller than one). However, as pointed out recently by Abdelgalil and Georgiou \cite{Abdelgalil_Georgiou_2024}, this is not enough. Indeed, the crux of the argument in \cite{Brockett_2007} is that the above contraction condition implies that the map
$$
e^{-At}K_t = {\rm id} + W(0,t)W(0,T)^{-1} \big(e^{-AT}\psi - {\rm id}\big)
$$
is invertible for every $t$. This, in turn, relies on the claim that, since $W(0,T) - W(0,t)$ is positive definite for every $0 < t < T$ by controllability, it follows that $\|W(0,t)W(0,T)^{-1}\| \le 1$. However, the latter claim is not valid: Since the matrix $W(0,t)W(0,T)^{-1}$ need not be symmetric, we can only guarantee that its spectral radius is bounded above by one. Following the line of argument in \cite{Abdelgalil_Georgiou_2024}, we can easily show that Brockett's argument goes through if we modify his contraction condition by replacing $\psi$ with 
	\begin{align}\label{eq:psi_hat}
		\hat{\psi}(x) := W(0,T)^{-1/2}e^{-AT}\psi\big(W(0,T)^{1/2}x\big).
	\end{align}
From this, it follows readily that the mapping
\begin{align*}
	\tilde{K}_t := {\rm id}  + W(0,T)^{-1/2}W(0,t)W(0,T)^{-1/2} \big(\hat{\psi} - {\rm id}\big)
\end{align*}
is invertible since $\|W(0,T)^{-1/2}W(0,t)W(0,T)^{-1/2}\| \le 1$ (multiply the matrix inequality $W(0,t) \le W(0,T)$ on the left and on the right by $W(0,T)^{-1/2}$). The desired conclusion then follows from the fact that 
\begin{align*}
	K_t(x) = e^{At}W(0,T)^{1/2}\tilde{K}_t(W(0,T)^{-1/2}x).
\end{align*}

The above assumption on $\psi$  is fairly restrictive. For an arbitrary $\psi$, one could first represent it as a composition $\psi_k \circ \dots \circ \psi_1$ such that each $\psi_i$ satisfies (modified) Brockett's condition and then apply the above construction for each $i$. However, the number $k$ will, in general, be very large, resulting in controls of very high complexity (we will come back to the issue of complexity later in the broader context of nonlinear systems). It was later shown by Hindawi et al.~\cite{Hindawi_2011} and Chen et al.~\cite{Chen_2017} that controllability is sufficient for any $\psi$ that can be written as the gradient of a convex function after a certain change of coordinates. These results can also be phrased in the setting of optimal transportation, as the resulting constructions are optimal for quadratic costs of the form $L(x,u) = x^\trn Qx + u^\trn R u$ with symmetric positive-semidefinite $Q$ and symmetric positive-definite $R$ \cite{Hindawi_2011}.

Here, we will give a result of the same flavor that preserves the spirit of Brockett's proof. The key concept we will need is that of a \textit{monotone mapping} \cite{Rockafellar_1998}: A mapping $\psi : \Reals^n \to \Reals^n$ is monotone if
\begin{align*}
	\langle x - y, \psi(x) - \psi(y) \rangle \ge 0, \qquad \text{for all } x,y \in \Reals^n.
\end{align*}
Note that this definition makes no assumptions on differentiability of $\psi$ (in fact, it can be extended to set-valued mappings); however, when $\psi$ is differentiable, monotonicity is equivalent to the symmetric part of the Jacobian $D\psi$ being everywhere positive-semidefinite. In particular, any $\psi$ given by the gradient of a convex function is monotone. 

\begin{theorem}\label{thm:linear} If the system \eqref{eq:LTI} is controllable and if  $\psi : \Reals^n \to \Reals^n$ is such that the mapping $\hat{\psi}$ defined in \eqref{eq:psi_hat} is monotone, then there exists a feedback control law $u(t,x)$ that steers the state of \eqref{eq:LTI} from $x(0)=x_0$ to $x(T)=\psi(x_0)$ for every $x_0 \in \Reals^n$.
\end{theorem}

\begin{proof} We will show that the map $K_t$ defined in \eqref{eq:Kt} is injective for $0 \le t < T$, thus the feedback control $u(t,x)$ given by \eqref{eq:ut_x} will do the job. 
	
To that end, we claim that, for any $0 < t < T$,
\begin{align}\label{eq:K_t_injective}
	x_0 \neq y_0 \, \Rightarrow \, \langle x_0 - y_0, W(0,t)^{-1} e^{-At}(K_t(x_0) - K_t(y_0))\rangle > 0,
\end{align}
which will imply that $K_t$ is injective since the matrices $W(0,T)W(0,t)^{-1}e^{-At}$ are nonsingular for $t > 0$. Using the definition of $\hat{\psi}$ in \eqref{eq:psi_hat}, we can write
\begin{align*}
	K_t(x_0) &= e^{At}\Big(\big(I-W(0,t)W(0,T)^{-1}\big)x_0 \\
	& \qquad \qquad + W(0,t)W(0,T)^{-1/2}\hat{\psi}(W(0,T)^{-1/2}x_0) \Big).
\end{align*}
For any $x_0,y_0$, let $\hat{x}_0 := W(0,T)^{-1/2}x_0$ and $\hat{y}_0 := W(0,T)^{-1/2}y_0$. Then
\begin{align*}
&	\langle x_0 - y_0, W(0,t)^{-1} e^{-At}(K_t(x_0) - K_t(y_0))\rangle \nonumber\\
&= \langle x_0 - y_0, (W(0,t)^{-1}-W(0,T)^{-1})(x_0 - y_0)\rangle   + \langle \hat{x}_0 - \hat{y}_0, \hat{\psi}(\hat{x}_0) - \hat{\psi}(\hat{y}_0) \rangle.
\end{align*}
If $x_0 \neq y_0$, then the first term on the right-hand side is strictly positive  by controllability, while the second term is nonnegative since $\hat{\psi}$ is monotone. This proves \eqref{eq:K_t_injective}, so that $K_t$ is indeed injective  and $u(t,x)$ in \eqref{eq:ut_x} gives the desired feedback control law. 
\end{proof}

To connect Theorem~\ref{thm:linear} to the problem of controlling a given initial density $\rho_0$ to a given final density $\rho_1$ using \eqref{eq:LTI}, let $\hat{\rho}_0$ be the pushforward of $\rho_0$ by the invertible linear transformation $x \mapsto W(0,T)^{1/2}x$ and, similarly, let $\hat{\rho}_1$ be the pushforward of $\rho_1$ by the invertible linear transformation $x \mapsto W(0,T)^{-1/2}e^{-AT}x$. By Brenier's theorem, there is a monotone map $\hat{\psi}$ such that $\hat{\rho}_1 = \hat{\psi}_*\hat{\rho}_0$. Then the map $\psi$ related to $\hat{\psi}$ via \eqref{eq:psi_hat} evidently pushes $\rho_0$ forward to $\rho_1$, and in that case Theorem~\ref{thm:linear} tells us how to construct the desired feedback control $u(t,x)$. In particular, the Benamou--Brenier dynamic formulation of optimal transportation is a special case corresponding to $A = 0$ and $B = I$ with $m = n$, cf.~\cite{Hindawi_2011,Chen_2017}.

\section{Controllability on the group of diffeomorphisms}
\label{sec:diffeo}

Let us now consider the case of a nonlinear system
\begin{align}\label{eq:nonlinear_system}
\dot{x} = f(x,u)
\end{align}
whose state space $M$ is a smooth (say, $C^\infty$) closed finite-dimensional manifold. A smooth diffeomorphism $\psi : M \to M$ is given, and the problem is to determine whether there exists a control law that can steer every initial state $x(0) = x_0 \in M$ to $x(T) = \psi(x_0)$, for a fixed finite $T > 0$. By analogy with the linear case, we hope to capitalize on some form of controllability of \eqref{eq:nonlinear_system}. Here, however, apart from the usual complications arising in the context of nonlinear controllability \cite{Hermann_1977}, a major difficulty is the lack of explicit expressions for control laws that transfer a given initial state to a given final state. Nevertheless, we can still aim for a structural result of some form.

One such result was obtained by Agrachev and Caponigro \cite{Agrachev_2009}. We will discuss it shortly in full generality, but for now we mention its corollary for driftless control-affine systems of the form
\begin{align}\label{eq:control_affine}
	\dot{x} = \sum^m_{i=1}u_i f_i(x),
\end{align}
where $f_1,\dots,f_m$ are smooth vector fields on $M$ and $u_1,\dots,u_m$ are  real-valued control inputs. Suppose that $\{f_1,\dots,f_m\}$ is a bracket-generating family, i.e., for each $x \in M$ the set $\{ g(x) : g \in {\rm Lie}(f_1,\dots,f_m) \}$, where ${\rm Lie}(f_1,\dots,f_m)$ is the Lie algebra generated by $f_1,\dots,f_m$, coincides with the tangent space $T_xM$ to $M$ at $x$. Then, for any diffeomorphism $\psi : M \to M$ isotopic to the identity there exists a time-dependent feedback control law $u(t,x) = (u_1(t,x),\dots,u_m(t,x))$ that transfers $x(0) = x_0$ to $x(1) = \psi(x_0)$ for every $x_0 \in M$. (Two diffeomorphisms $\psi,\psi' : M \to M$ are \textit{isotopic} if there exists a smooth map $H : [0,1] \times M \to M$, such that $H(t,\cdot) : M \to M$ for each $t \in [0,1]$ is a diffeomorphism, $H(0,\cdot) = \psi(\cdot)$, and $H(1,\cdot) = \psi'(\cdot)$ \cite{Banyaga_1997}. We will denote by ${\rm Diff}_0(M)$ the family of all diffeomorphisms on $M$ that are isotopic to the identity.) 

The general setting considered in \cite{Agrachev_2009} is as follows: Let a family $\cF$ of smooth vector fields on $M$ be given. The \textit{control group} of $\cF$, defined by
\begin{align*}
	{\mathbb G}(\cF) := \left\{ e^{t_kf_k} \circ \dots \circ e^{t_1f_1} : t_i \in \Reals,\, f_i \in \cF,\, k \in \Naturals \right\},
\end{align*}
is a subgroup of the group ${\rm Diff}(M)$ of smooth diffeomorphisms of $M$. Suppose that ${\mathbb G}(\cF)$ acts transitively on $M$, i.e., for any $x,y \in M$ there exists some $\psi \in {\mathbb G}(\cF)$ such that $y = \psi(x)$. Then we have the following \cite{Agrachev_2009}:

\begin{theorem}\label{thm:Agrachev_Caponigro} There exist a neighborhood $\cO$ of the identity in ${\rm Diff}_0(M)$ and a positive integer $k$ that depends only on $\cF$, such that every $\psi \in \cO$ can be represented as
	\begin{align}\label{eq:Agrachev_Caponigro}
		\psi = e^{a_k f_k} \circ \dots \circ e^{a_1 f_1}
	\end{align}
	for some $f_1,\dots,f_k \in \cF$ and some $a_1,\ldots,a_k \in C^\infty(M)$.
\end{theorem}
Note that the vector fields in \eqref{eq:Agrachev_Caponigro} are rescaled by smooth real-valued functions on $M$, which explains the origin of time-varying feedback controls $u_i(t,x)$ in the above discussion of control-affine systems. It is also evident that these controls will be piecewise constant in $t$.

With this in mind, let us consider the system \eqref{eq:control_affine}, where $\cF = \{f_1,\dots,f_m\}$ is bracket-generating. Let $\psi \in {\rm Diff}_0(M)$ be given. We will apply Theorem~\ref{thm:Agrachev_Caponigro} to $\cF$---because $\cF$ is bracket-generating, ${\mathbb G}(\cF)$ acts transitively on $M$. Following \cite{Caponigro_2011}, we first establish a fragmentation property of $\psi$ relative to $\cO$, i.e., show that there exist some $\psi_1,\dots,\psi_N \in \cO$ such that $\psi = \psi_N \circ \dots \circ \psi_1$.  Since $\psi \in {\rm Diff}_0(M)$, there exists a smooth path $[0,1] \ni t \mapsto \varphi_t \in {\rm Diff}_0(M)$ such that $\varphi_1 = \psi$ and $\varphi_0 = {\rm id}$. For every $N \in \Naturals$, the maps
\begin{align*}
	\psi_i := \varphi_{i/N} \circ \varphi_{(i-1)/N}^{-1}, 
	\qquad i=1,\dots,N
\end{align*}
belong to ${\rm Diff}_0(M)$, and $\psi = \psi_N \circ \psi_{N-1} \circ \dots \circ \psi_1$. We can ensure that each $\psi_i \in \cO$ by choosing $N$ large enough. Applying Theorem~\ref{thm:Agrachev_Caponigro} to each $\psi_i$, we conclude that there exist $K = kN$ smooth functions $a_1,\dots,a_K \in C^\infty(M)$, such that $\psi$ can be represented as
\begin{align*}
	\psi = e^{a_K f_{i_K}} \circ \dots \circ e^{a_1 f_{i_1}},
\end{align*}
for some choice of indices $i_1,\dots,i_K \in \{1,\dots,m\}$, cf.~\cite[Prop.~4]{Caponigro_2011}. From this representation, it is straightforward to derive $m$ feedback controls $u_i : [0,1] \times M \to \Reals$ that are piecewise constant in $t$, with the number of pieces (switchings) equal to $K$. Thus, the integer $k$ in Theorem~\ref{thm:Agrachev_Caponigro} is a lower bound on the number of switchings, which in turn is a natural measure of implementation complexity of a control law. Since the result of \cite{Agrachev_2009} has been used as a black-box device in subsequent works \cite{Caponigro_2011,Elamvazhuthi_2023}, it is of interest to provide some quantitative estimates of $k$. 

To that end, we first make some assumptions on $M$ and $\cF$. We take $M$ to be a smooth compact manifold of dimension $n$ isometrically embedded in $\Reals^d$ for some $d > n$, and we will equip $M$ with the ambient metric $d(x,y) = |x-y|$, where $|\cdot|$ is the Euclidean norm on $\Reals^d$. We further assume that $M$ has \textit{positive reach} $\tau > 0$, where the reach of a set $A \subset \Reals^d$ is defined as the largest value of $\tau$, such that any point at a distance $0 < r < \tau$ from $A$ has a unique nearest point in $A$ \cite{Federer_1959}. For example, the unit sphere $S^{d-1}$ has reach $1$. We will assume that the vector fields $f \in \cF$ are uniformly bounded in the $C^1(M)$ seminorm, in the following sense \cite[Sec.~2.2]{Agrachev_2004}. For $r = 0,1,\dots$, the $C^r(M)$ seminorm of a function $a \in C^\infty(M)$ is defined by
\begin{align*}
	\|a\|_{C^r(M)} := \sup \left\{ |D_{i_1} \dots D_{i_\ell}a(x)| : x \in M,\, 1 \le i_1,\dots,i_\ell \le d,\, 0 \le \ell \le r \right\},
\end{align*}
where $D_1,\dots,D_d$ are the orthogonal projections on $M$ of the standard basis vector fields $\frac{\partial}{\partial x_1},\dots,\frac{\partial}{\partial x_d}$ on $\Reals^d$. The $C^r(M)$ seminorm of a vector field $f$ on $M$ is defined as
\begin{align*}
	\| f \|_{C^r(M)} := \sup \left\{ \| fa \|_{C^r(M)} : a \in C^\infty(M),\, \|a\|_{C^{r+1}(M)} = 1 \right\},
\end{align*}
where $fa \in C^\infty(M)$ is the Lie derivative of $a$ along $f$. In terms of these definitions, we will assume that
\begin{align*}
	 \sup_{f \in \cF} \|f\|_{C^1(M)} < \infty.
\end{align*}
We then have the following:

\begin{theorem}\label{thm:AC_quant} Let $M$ and $\cF$ satisfy the above assumptions. Then Theorem~\ref{thm:Agrachev_Caponigro} holds with 
	\begin{align}\label{eq:AC_lower_bound}
 k \ge \frac{{\rm vol}(M)}{{\rm vol}(B_{\Reals^n}(0,1))} \frac{n^2}{16^n} \left(\frac{1}{r}\right)^n
	\end{align}
	for some sufficiently small $r \in (0,\tau)$, where $B_{\Reals^n}(0,1)$ is the $n$-dimensional Euclidean ball of radius $1$ centered at the origin.
\end{theorem}
\begin{remark} The lower bound on $k$ in \eqref{eq:AC_lower_bound} is not optimal, but rather an artifact of the specific construction used in \cite{Agrachev_2009}. Moreover, as will become evident from the proof of the theorem, the actual number $k$ will be much larger than the right-hand side of \eqref{eq:AC_lower_bound}. At any rate, the exponential dependence of the lower bound on the dimension of the state space $M$ should be kept in mind when using the result of \cite{Agrachev_2009}.
\end{remark}

\subsection{The proof of Theorem~\ref{thm:AC_quant}}

We will follow the logic of \cite{Agrachev_2009}, but with some of the steps replaced by more explicit quantitative arguments.

For any subset $V \subseteq \Reals^n$ containing the origin, we will denote by $C^\infty_0(V,\Reals^k)$ the space of $C^\infty$ functions $F : V \to \Reals^k$, such that $F(0)=0$. When $k = 1$, we will simply write $C^\infty_0(V)$. Also, $B_{\Reals^k}(z,r)$ will denote the  Euclidean ball of radius $r$ centered at $z \in \Reals^k$.

The first key ingredient in the proof of Theorem~\ref{thm:Agrachev_Caponigro} is the following:

\begin{proposition}\label{prop:AC_aux} Let $X_1,\dots,X_n$ be vector fields on $\Reals^n$ with $\|X_i\|_{C^1(\Reals^n)} < \infty$ for all $i$, such that $X_1(0),\dots,X_n(0)$ are linearly independent. Then there exist constants $r,\eps > 0$ and an open set
	\begin{align*}
		\cV \subseteq \left\{ F \in C^\infty_0(B_{\Reals^n}(0,r),\Reals^n) : \| F - {\rm id} \|_{C^1} < \eps \right\},
	\end{align*}
	such that every $F \in \cV$ can be written as
	\begin{align*}
		F = e^{a_n X_n} \circ \dots \circ e^{a_1X_1}\big|_{B_{\Reals^n}(0,r)}
	\end{align*}
	for some $a_1,\dots,a_n \in C^\infty_0(B_{\Reals^n}(0,r))$.
\end{proposition}
\begin{remark} By multiplying each $X_i$ by a suitable bump function, the everywhere $C^1$ boundedness assumption on $X_1,\dots,X_n$ can be relaxed to $C^1$ boundedness on a ball of sufficiently large radius centered at the origin.
\end{remark}

The proof of Proposition~\ref{prop:AC_aux} makes use of a lemma which states that, for some $r > 0$, we can represent every $F$ in a certain open neighborhood of the identity in $C^\infty_0(B_{\Reals^n}(0,r),\Reals^n)$ as a composition of $n$ smooth maps with certain additional properties. We give a quantitative version here; given the lemma, Proposition~\ref{prop:AC_aux} can be proved exactly as in \cite{Agrachev_2009}.
 
\begin{lemma}\label{lm:AC_aux} Let $X_1,\dots,X_n$ be vector fields on $\Reals^n$ satisfying the conditions of Proposition~\ref{prop:AC_aux}. Let $\cU_0$ be a given neighborhood of the identity in $C^\infty_0(\Reals^n,\Reals^n)$. Then there exist constants $r,\eps > 0$, such that any $F \in C^\infty_0(B_{\Reals^n}(0,r),\Reals^n)$ with $\| F - {\rm id}\|_{C^1} < \eps$ can be represented as a composition of the form
	\begin{align*}
		F = \psi_n \circ \psi_{n-1} \circ \dots \circ \psi_1 \big|_{B_{\Reals^n}(0,r)}
	\end{align*}
for some $\psi_1,\dots,\psi_n \in \cU_0$, where each $\psi_i$ preserves the integral curves of $X_i$, i.e., for any $x \in B_{\Reals^n}(0,r)$ there exists some $t_i = t_i(x) \in \Reals$ such that $\psi_i(x) = e^{t_iX_i}(x)$.
\end{lemma}

\begin{proof}[Proof (of Lemma~\ref{lm:AC_aux})] We first establish a quantitative inverse function theorem following the ideas in \cite[Sec.~8]{Christ_1985}; one could also use a less direct argument appealing to the Lyusternik--Graves theorem \cite{Dontchev_2009}, with similar estimates.  For every $x \in \Reals^n$, let $\bX(x)$ be the $n \times n$ matrix with columns $X_n(x),X_{n-1}(x),\dots,X_1(x)$. Since $X_1(0),\dots,X_n(0)$ span $\Reals^n$, the smallest singular value of $\bX(0)$ is positive:
	\begin{align*}
		\sigma_n(\bX(0)) = c > 0.
	\end{align*}
The function $x \mapsto \sigma_n(\bX(x))$ is continuous as a consequence of the continuity of $X_i$'s and Weyl's perturbation theorem \cite{Bhatia_1997}. Therefore, there exists some $r > 0$, such that
\begin{align*}
	\sigma_n(\bX(x)) \ge c/2, \qquad x \in V := B_{\Reals^n}(0,r).
\end{align*}
Consider the function $G : \Reals^n \times \Reals^n \to \Reals^n$ given by
\begin{align*}
	G(x,y) := e^{y_n X_n} \circ \dots \circ e^{y_1 X_1}(x).
\end{align*}
Then, denoting by $D_1G$ and $D_2G$ the partial derivatives of $G$ w.r.t.\ $x$ and $y$ respectively, we have
\begin{align*}
	D_2 G (x,0) = \bX(x).
\end{align*}
Hence, there exists some $\eps_0 > 0$, such that
\begin{align*}
	\sigma_n\left(D_2 G(x,\bar{y})\right) \ge c/4, \qquad x \in V,\, \bar{y} \in W:= B_{\Reals^n}(0,\eps_0)
\end{align*}
so, by the classical inverse function theorem, the map $y \mapsto G(x,y)$ is invertible on $W$ for each $x \in V$. We claim that, for any $0 \le \eps \le \eps_0$, the image of the ball $B_{\Reals^n}(0,\eps)$ under $G(x,\cdot)$ contains the ball $B_{\Reals^n}(x,c\eps/4)$, and thus for any $z \in B_{\Reals^n}(x,c\eps/4)$ there exists a unique $y \in B_{\Reals^n}(0,\eps)$, such that $G(x,y) = z$. To establish the claim, fix a unit vector $v \in \Reals^n$ and define a vector field $Y$ on $W$ by $D_2G(x,y)Y(y) = v$ for all $y \in W$. Since $|Y(y)| \le 4/c$ on $W$,  the curve $\xi(t) := e^{tY}(0)$ is well-defined and remains in $B_{\Reals^n}(0,\eps)$ for all $0 \le t \le c\eps/4$, as long as $\eps \le \eps_0$. Moreover, since by the chain rule we have
\begin{align*}
	\frac{\dif}{\dif t}G(x,\xi(t)) = D_2 G(x,\xi(t))Y(\xi(t)) = v,
\end{align*}
it follows that $G(x,\xi(t)) = x + tv \in B_{\Reals^n}(x,c\eps/4)$ for $0 \le t \le c\eps/4$. Since $v$ was arbitrary, the claim follows.

Now, fix some $\eps > 0$ to be chosen later and let $\cU_\eps$ be the set of all $F \in C^\infty_0(V,\Reals^n)$, such that
\begin{align}\label{eq:F_C1_eps}
\| F - {\rm id} \|_{C^1} < c\eps/4.
\end{align}
By the preceding discussion, if $\eps \le \eps_0$, then for any $F \in \cU_\eps$ and for any $x \in V$ there exists a unique $y(x) \in B_{\Reals^n}(0,\eps) \subseteq W$, such that $F(x) = G(x,y(x))$. The chain rule then gives
\begin{align*}
	DF(x) &= D_1G(x,y(x)) + D_2G(x,y(x)) Dy(x),
\end{align*}
and, since $D_2G(x,y(x))$ is invertible for $x \in V$, we can solve for $Dy(x)$:
\begin{align*}
	Dy(x) = D_2G(x,y(x))^{-1}\big(DF(x)-D_1G(x,y(x))\big).
\end{align*}
Consider the curve $\eta(t)$, $0 \le t \le n$, given by
\begin{align*}
	\dot{\eta}(t) = Z(\eta(t),t), \qquad \eta(0) = x
\end{align*}
where, for $0 \le t \le n$,
\begin{align*}
	Z(x,t) := y_i X_i(x), \qquad i - 1 \le t < i,\, i = 1,\dots,n.
\end{align*}
Then $G(x,y) = \eta(n)$, and we can express the derivative $D_1G(x,y)$ as the $t=n$ solution of the matrix-valued variational equation
\begin{align*}
	\dot{\Lambda}(t) = \frac{\partial Z}{\partial x}(\eta(t),t) \Lambda(t), \qquad \Lambda(0) = I.
\end{align*}
For $|y|$ small enough, we will have
\begin{align*}
	\| D_1G(x,y) - I \| \le C|y|
\end{align*}
for some constant $C$ that depends only on $n$ and on the $C^1$ seminorms of the $X_i$'s. It follows that, for $F \in \cU_\eps$ with $\eps$ small enough, we will have
\begin{align*}
	\sup_{x \in V} \left\| Dy(x) \right\| \le \frac{4}{c}\left(\frac{c}{4}+C\right)\eps =: C'\eps,
\end{align*}
i.e., the map $x \mapsto y(x)$ is $C'\eps$-Lipschitz.

Now consider, following \cite{Agrachev_2009}, the maps
\begin{align*}
	\Phi_i(x) := e^{y_i(x)X_i} \circ \dots \circ e^{y_1(x)X_1}(x), \qquad i = 0,\dots,n
 \end{align*}
where $\Phi_0 = {\rm id}$ and $\Phi_n = F$. Since $\Phi_i(x) = G(x,(y_1(x),\dots,y_i(x),0,\dots,0))$, it is a diffeomorphism with domain $V$ by the preceding discussion, and, in fact, for $\eps$ small enough we will have $\| \Phi_i - {\rm id}\|_{C^1} \le C''\eps$ for some constant $C'' > 0$. Finally, define, for $i = 1,\dots,n$, the maps
\begin{align*}
	\psi_i(x) := e^{y_i(\Phi^{-1}_{i-1}(x))X_i}(x),
\end{align*}
so that $F = \psi_n \circ \dots \circ \psi_1$. Each $\psi_i$ evidently preserves the integral curves of $X_i$. Moreover, from \eqref{eq:F_C1_eps} it follows that $|y_i(\Phi^{-1}_{i-1}(x))| \le \eps$ for all $x \in V$, and thus we can ensure that each $\psi_i \in \cU_0$ by taking $\eps$ sufficiently small. 
\end{proof}

The second key ingredient is a fragmentation lemma for the elements of ${\rm Diff}_0(M)$ along the lines of \cite[Lemma~3.1]{Palis_1970}, see also \cite[Lemma~2.1.8]{Banyaga_1997}. Again, we give a quantitative version. Recall  the definition of the \textit{support} of a diffeomorphism $\psi \in {\rm Diff}(M)$:
\begin{align*}
	{\rm supp}(\psi) := \overline{\left\{ x \in M : \psi(x) \neq x \right\}}.
\end{align*}
In other words, the support of $\psi$ is the closure of the set on which $\psi$ differs from the identity \cite{Banyaga_1997}.

\begin{lemma}\label{lm:AC_frag} Let $\cO$ be a neighborhood of the identity in ${\rm Diff}(M)$. Then, for any $r \in (0,\tau)$, where $\tau$ is the reach of $M$, there exist $N$ points $z_1,\dots,z_N \in M$ with
	\begin{align}\label{eq:covering_number}
	\frac{{\rm vol}(M)}{{\rm vol}(B_{\Reals^n}(0,1))}n\left(\frac{1}{16}\right)^n \left(\frac{1}{r}\right)^n	\le N \le \frac{{\rm vol}(M)}{{\rm vol}(B_{\Reals^n}(0,1))} n\left(\frac{\pi}{2}\right)^n \left(\frac{1}{r}\right)^n,
	\end{align}
such that the sets $U_i = M \cap B_{\Reals^d}(z_i,r)$ cover $M$ and such that any $\psi \in \cO \cap {\rm Diff}_0(M)$ can be written in the form $\psi = \psi_N \circ \dots \circ \psi_1$, where each $\psi_i$ is an element of $\cO$ with ${\rm supp}(\psi_i) \subseteq U_i$ for all $i$.
\end{lemma}
\begin{remark} Since the group ${\rm Diff}_0(M)$ is path-connected, it is generated by any neighborhood $\cO$ of the identity. Thus, ${\rm Diff}_0(M)$ is generated by the set $\{ \psi \in \cO: {\rm supp}(\psi) \subseteq U_i \text{ for some $1 \le i \le N$}\}$.\end{remark}

\begin{proof} Since $M$ is compact, for any $r > 0$ it can be covered by finitely many sets of the form $U_i = M \cap B_{\Reals^d}(z_i,r)$ for some $z_1,\dots,z_N \in M$. From \cite[Corollary~10]{Block_2022}, the smallest value of $N$ for a given $r < \tau$, i.e., the covering number of $M$ at resolution $r$ w.r.t.\ the ambient metric, can be estimated from above and from below as in \eqref{eq:covering_number}. Let $U_1,\dots,U_N$ be such a covering of $M$. We can now follow the steps in the proof of \cite[Lemma~5.4]{Agrachev_2009} to finish. 
\end{proof}

After these preparations, we can proceed essentially along the same lines as in \cite{Agrachev_2009}. Let $\cP := {\mathbb G}(\tilde{\cF})$, the control group generated by
\begin{align*}
	\tilde{\cF} := \{af : a \in C^\infty(M),\, f \in \cF\},
\end{align*}
and define the isotropy subgroup of $\cP$ at $x \in M$:
\begin{align*}
	\cP_x := \{ \psi \in \cP : \psi(x) = x \}.
\end{align*}
For an open set $U \subseteq M$ and $x \in U$, we will denote by $C^\infty_x(U,M)$ the family of all smooth maps $F : U \to M$, such that $F(x) = x$. 

\begin{lemma}\label{lm:AC_nbhd} There exists a constant $r \in (0,\tau)$ such that, for any $x \in M$ and for $U := M \cap B_{\Reals^d}(x,r)$, the set
	\begin{align}\label{eq:AC_nbhd}
		\{ \psi|_U : \psi \in \cP_x \}
	\end{align}
has nonempty interior in $C^\infty_x(U,M)$ which contains the identity. Moreover, there exist $k \ge n$ vector fields $f_1,\dots,f_k$ depending on $x$, such that any element of the interior of \eqref{eq:AC_nbhd} can be represented as a composition of the form
\begin{align}\label{eq:AC_comp}
	e^{a_k f_k} \circ \dots \circ e^{a_1 f_1}
\end{align}
for some $a_1,\dots,a_k \in C^\infty(M)$.
\end{lemma}
\begin{proof} We follow the proof of Lemma~5.1 and Corollary~5.2 in \cite{Agrachev_2009}, but with some steps made more explicit. Since ${\mathbb G}(\cF)$ acts transitively on $M$, Sussmann's orbit theorem \cite{sussmann_1973} implies that the tangent space $T_xM$ to $M$ at $x$ is equal to the span of the set $\{\psi_*f(x) : \psi \in {\mathbb G}(\cF), f \in \cF\}$, where $\psi_*f$ denotes the action of the tangent map of $\psi$ on $f$. Thus, there exist vector fields $X_1,\dots,X_n$ of the form $X_i = (\psi_i)_*f_i$ for some $\psi_i \in {\mathbb G}(\cF)$ and $f_i \in \cF$, such that $X_1(x),\dots,X_n(x)$ span $T_xM$. Moreover, for all smooth functions $a_1,\dots,a_n \in C^\infty(M)$ that vanish at $x$, the diffeomorphism
	\begin{align}\label{eq:x_diffeo}
		\begin{split}
	&	e^{a_nX_n} \circ \dots \circ e^{a_1X_1} \\
	& \qquad = \psi_n \circ e^{(a_n \circ \psi_n)f_n} \circ \psi^{-1}_n \circ \dots \circ \psi_1 \circ e^{(a_1 \circ \psi_1)f_1} \circ \psi^{-1}_1
	\end{split}
	\end{align}
	belongs to the isotropy group $\cP_x$. Now, using the techniques of \cite{Fefferman_2016}, we can show that, for any $r > 0$ small enough, the set $U = M \cap B_{\Reals^d}(x,r)$ can be locally coordinatized using the elements of $T_xM$ as
	\begin{align*}
		U = \{ x + v + \Psi(v) : v \in B_{\Reals^n}(0,r') \} \cap B_{\Reals^d}(x,r)
	\end{align*}
	for some $r' > r$, where $\Psi$ is a $C^2$ map from $B_{\Reals^n}(0,r')$ into $\Reals^{d-n}$, such that $\Psi(0) = 0$, $D\Psi(0) = 0$, and the Jacobian $D\Psi$ is $\frac{C}{\tau}$-Lipschitz, where $C$ is some constant that depends only on $n$. In particular, the map $v \mapsto x + v + \Psi(v)$ is invertible. (In the terminology of \cite{Fefferman_2016}, the set
	$$
	\{ x + v + \Psi(v) : v \in B_{\Reals^n}(0,r') \}
	$$
is a patch over $T_xM$ of radius $r'$, centered at $x$ and tangent to $T_xM$ at $x$.) Hence, by reparametrizing and rescaling, we end up in the situation of Proposition~\ref{prop:AC_aux}. It then follows that the set \eqref{eq:AC_nbhd} has nonempty interior in $C^\infty_x(U,M)$. Moreover, since the vector fields in $\cF$ are uniformly bounded in $C^1$ seminorm, it follows from the proof of Lemma~\ref{lm:AC_aux} that the constant $r$ can be chosen independently of $x$. 
	
To show that the interior of the set \eqref{eq:AC_nbhd} contains the identity map, we let $\cA$ be an open subset of $C^\infty_x(U,M)$ contained in the interior of \eqref{eq:AC_nbhd}. Then, for any $\psi_0$ such that $\psi_0|_U \in \cA$, the set $\cO := \psi^{-1}_0 \circ \cA$ is a neighborhood of the identity contained in \eqref{eq:AC_nbhd}.

Finally, since each $\psi_i$ in \eqref{eq:x_diffeo} belongs to the control group ${\mathbb G}(\cF)$, it can be expressed as
\begin{align*}
	\psi_i = e^{t_{i,k_i}f_{i,k_i}} \circ \dots e^{t_{i,1} f_{i,1}}
\end{align*}
for some integer $k_i \ge 1$, times $t_{i,1},\dots,t_{i,k_i} \in \Reals$, and vector fields $f_{i,1},\dots,f_{i,k_i} \in \cF$. Substituting this into \eqref{eq:x_diffeo} and relabeling the functions $a$ and the vector fields $f$ as needed, we arrive at the representation \eqref{eq:AC_comp} for some $a_1,\dots,a_k \in C^\infty(M)$ and $f_1,\dots,f_k \in \cF$, where
\begin{align*}
	k = n + 2\sum^n_{i=1}k_i \ge n.
\end{align*}
This completes the proof. 
\end{proof}

\begin{lemma}\label{lm:AC_isotropy} Let $\cO$ be a neighborhood of the identity in ${\rm Diff}(M)$. Then for any $x \in M$ and any open set $U \subset M$ containing $x$,
	\begin{align*}
		x \in {\rm int}\{\psi(x) :  \psi \in \cO \cap \cP, {\rm supp}(\psi) \subset U \}.
	\end{align*}
\end{lemma}
\begin{proof} See the proof of Lemma~5.3 in \cite{Agrachev_2009}. 
\end{proof}

Now, let $\{U_i\}_{1 \le i \le N}$ be a finite covering of $M$ by sets of the form $U_i = M \cap B_{\Reals^d}(z_i,r)$ for some $z_1,\dots,z_N \in M$, where $r > 0$ is the constant from Lemma~\ref{lm:AC_nbhd}. By Lemma~\ref{lm:AC_nbhd}, for each $i$ there exists a neighborhood $\cO_i$ of the identity in ${\rm Diff}(M)$, such that any $\psi \in \cO_i$ with ${\rm supp}(\psi) \subseteq U_i$ belongs to $\cP$. Moreover, by Lemma~\ref{lm:AC_isotropy} we may assume that $\psi(z_i) = z_i$, i.e., $\psi \in \cP_{z_i}$. Taking $\cO = \cup_{1 \le i \le N}\cO_i$, Lemma~\ref{lm:AC_frag} says that every $\psi \in \cO$ can be written as $\psi = \psi_N \circ \dots \circ \psi_1$, where each $\psi_i \in \cO_i$. This, together with the lower estimate for $N$ in \eqref{eq:covering_number}, completes the proof of Theorem~\ref{thm:AC_quant}. The extra factor of $n$ in \eqref{eq:AC_lower_bound} comes from Lemma~\ref{lm:AC_nbhd}. Moreover, the lower bound in \eqref{eq:AC_lower_bound} is rather generous since it only accounts for the lower bound on the number of fragments of $\psi$ according to Lemma~\ref{lm:AC_frag} and for the fact that $k \ge n$ in Lemma~\ref{lm:AC_nbhd}.

\section{Conclusion}
\label{sec:conclusion}

While we were able to touch upon only some of the aspects of the controllability problem for the Liouville equation, we can highlight certain key ideas: First of all, the use of time-varying feedback controls seems inevitable. Indeed, the situations where a given diffeomorphism can be realized as a finite-time flow map of \textit{some} time-invariant vector field are rather rare, as can be seen from the relative dearth of explicit results on the so-called \textit{embedding problem} \cite{Gronau_1991,Arango_2002}: Given a diffeomorphism $\psi : \Reals^n \to \Reals^n$, find a complete vector field $f$ on $\Reals^n$, such that $\psi(x) = e^f(x)$ for all $x \in \Reals^n$. Second, even if the requisite feedback controls can be obtained by switching between finitely many time-invariant flows, the resulting implementation complexity as measured by the number of switchings can be quite high, at least given the available techniques in the nonlinear setting, which rely on some combination of Sussmann's orbit theorem, smooth fragmentation, and covering arguments. Hence, the construction of more parsimonious controls is an interesting open problem. Another promising direction is to consider the possibility of feedback control in the space of densities, i.e., instead of state feedback of the form $u(t,x(t))$ use density feedback of the form $u(t,\rho(t,\cdot))$, where $\rho(t,\cdot)$ is the density of $x(t)$ \cite{Ambrosio_2009,Piccoli_2013}.

\section*{Acknowledgments}

The author would like to thank Joshua Hanson, Borjan Geshkovski, Yury Polyanskiy, and an anonymous reviewer for their comments on the manuscript, and Tryphon Georgiou for bringing to his attention the flaw in the controllability argument in \cite{Brockett_2007} (and consequently in an earlier version of this work). This work was supported by the NSF under awards CCF-2348624 (``Towards a control framework for neural generative modeling'') and CCF-2106358 (``Analysis and Geometry of Neural Dynamical Systems''), and  by the Illinois Institute for Data Science and Dynamical Systems (iDS${}^2$), an NSF HDR TRIPODS institute, under award CCF-1934986.

\bibliographystyle{alpha}
\bibliography{Liouville_arxiv.bbl}

\end{document}